\documentclass{amsart}

\usepackage{graphicx}
\usepackage{color}   
\usepackage{hyperref}
\hypersetup{
    colorlinks=true, 
    linktoc=all,     
    linkcolor=blue,  
}

\newtheorem{thm}{Theorem}[section]
\newtheorem{cor}[thm]{Corollary}
\newtheorem{lema}[thm]{Lemma}
\newtheorem{prop}[thm]{Proposition}

\theoremstyle{definition}
\newtheorem{defn}[thm]{Definition}

\theoremstyle{remark}
\newtheorem{rem}[thm]{Remark}

\def\supp{\mathop{\text{\normalfont supp}}}
\def\diver{\text{\normalfont div}}

\def\dist{\mathop{\text{\normalfont dist}}}

\numberwithin{equation}{section}
\newcommand{\R}{\mathbb R}

\newcommand{\N}{\mathbb N}

\newcommand{\D}{\displaystyle}

\def\C{\mathbf {C}}
\def\d{\mathbf {d}}

\def\LL{{\mathcal{L}}}

\newcommand{\intr}{\int_{\R^N}}

\newcommand{\ve}{\varepsilon}

\newcommand{\lam}{\lambda}

\begin{document}
\title[Limit of Orlicz Fractional Laplacians]{A H\"older Infinity Laplacian obtained as limit of Orlicz Fractional Laplacians}

\author[J. Fern\'andez Bonder, M. P\'erez-Llanos and A. Salort]{Juli\'an Fern\'andez Bonder, Mayte P\'erez-Llanos and Ariel M. Salort}

\address{Departamento de Matem\'atica, FCEyN - Universidad de Buenos Aires and
\hfill\break \indent IMAS - CONICET
\hfill\break \indent Ciudad Universitaria, Pabell\'on I (1428) Av. Cantilo s/n. \hfill\break \indent Buenos Aires, Argentina.}

\email[J. Fern\'andez Bonder]{jfbonder@dm.uba.ar}
\urladdr{http://mate.dm.uba.ar/~jfbonder}

\email[M. P\'erez-Llanos]{ maytep@dm.uba.ar}
\urladdr{http://mate.dm.uba.ar/~maytep }

\email[A. Salort]{asalort@dm.uba.ar}
\urladdr{http://mate.dm.uba.ar/~asalort}




\begin{abstract}
This paper concerns with the study of the asymptotic behavior of the solutions to a family of fractional type problems on a bounded domain, satisfying homogeneous Dirichlet boundary conditions. The family of differential operators includes the fractional $p_n$-Laplacian when $p_n\to\infty$ as a particular case, tough it could be extended to a function of the H\"older quotient of order $s$, whose primitive is an Orlicz function satisfying appropriated growth conditions.  The limit equation involves the H\"older infinity Laplacian.
\end{abstract}

\maketitle

\section{Introduction}

Asymptotic behavior of solutions of $p-$Laplacian type equations as $p\to \infty$ and its relation with the well-known $\infty-$Laplacian
\begin{equation}\label{infty-lapla}
-\Delta_\infty u := \left\langle D^2u \frac {Du}{|Du|},\frac{Du}{|Du|}\right\rangle,
\end{equation}
introduced by G. Aronsson \cite{A} when studying the {\it  Lipschitz extension problem}, have been extensively studied, giving rise to a considerable literature devoted to this issue. We suggest the interested reader  the survey paper on {\it Absolutely Minimizing Lipschitz functions} ({\it AMLE}) \cite{ACJ} and references therein.

However, there are indeed  several research directions assembled  behind this analysis. Another motivation comes from the analysis of {\it torsional creep problems}. According to \cite{BDM, K},  torsional creep problems  are related to inhomogeneous problems of the type
\begin{equation}\label{prob-p}
\begin{cases}
-\Delta_p u =f>0 &\text{in }\Omega,\\
u=0 &\text{on }\partial\Omega,
\end{cases}
\end{equation}
where, as usual, $-\Delta_p u := -\diver(|\nabla u|^{p-2}\nabla u)$ is the well-known $p-$laplacian operator.

As remarked in \cite{K}, several facts on elastic-plastic torsion theory suggested that, if we denote by $u_p$ the solution to \eqref{prob-p}, then necessarily, as $p\to\infty$, $u_p\to \dist(\cdot,\partial\Omega)$ in some sense, where $\dist(\cdot,\partial\Omega)$ stands for the distance function to the boundary of $\Omega$ with respect to the Euclidean norm $|\cdot|$, i.e. $\dist(x,\partial\Omega):=\inf_{y\in\partial\Omega}|x-y|$, for each $x\in\Omega$. Indeed, in \cite{K} the author established the uniform convergence of $u_p$ to  $\dist(\cdot,\partial\Omega)$ in $\overline\Omega$ as $p\to\infty$, via variational arguments and maximum principles, while in \cite{BDM} the authors used an approach based on the analysis of the viscosity solutions of the limiting problem of the family of equations \eqref{prob-p}.

 There can be considered more general differential operators on the left handside of \eqref{prob-p}. For example,  the $p-$laplacian operator can be replaced by its (inhomogeneous) variable exponent version, namely $-\Delta_{p(x)}u := -\diver(|\nabla u|^{p(x)-2}\nabla u)$, where the variable exponent $p(x)$ is a continuous function bounded away from 1 and $\infty$ and similar results are obtained. See for instance \cite{MRU,PL,PLR2,PLR}.

Different generalizations of this kind of problems arise when considering differential operators settled in Orlicz-Sobolev spaces, though the literature in this specific direction is much fewer in number, see  \cite{BM,PLR,S}.  Precisely, \cite{BM} is concerned with the asymptotic behavior of the sequence of solutions $u_n$ of
\begin{equation}\label{mm17}
\begin{cases}
-\Delta_{g_n} u := -\diver\left(\frac{g_n(|\nabla u|)}{|\nabla u|}\nabla u\right) = 1 & \text{in }\Omega,\\
u=0 & \text{on }\partial \Omega,
\end{cases}
\end{equation}
 when $g_n$ satisfies the so-called {\em Lieberman condition} (see \cite{Lieberman})
 $$
 p_n^- - 1\leq\frac{tg_n'(t)}{g_n(t)}\leq p_n^+-1 \quad \forall\;t\geq 0,
 $$
with $p_n^+ \leq \beta p_n^-$ for some $\beta>1$ and $p_n^-\to\infty$.

It was established in \cite{BM} that, if $u_n$ is the solution to \eqref{mm17}, then $u_n$ converges uniformly to $\dist(\cdot,\partial\Omega)$ in $\overline\Omega$, as $n\to\infty$.

We notice that it is a somehow expected result, given that problem \eqref{prob-p} is a particular case of \eqref{mm17}, just taking $g_n(t)=|t|^{p_n-2}t$ with $p_n\to\infty$.

Another field related to this kind of asymptotic  approach as $p\to\infty$ in equations of $p-$laplacian type are the so called {\it Tug-of-war-games}, see for example \cite{PSSW}. It is a two-person, zero-sum game, in which  a token is placed at some point $x_0\in\Omega$ and in each turn, the corresponding player moves the token to some position in a ball of radius $\varepsilon$, according to the result of a coin flip and the directions previously chosen by the players. The game finishes when one of the players gets out of the domain or achieves the boundary, thus receiving the pay-off value given by a function $g$ defined on the boundary $\partial\Omega$. In the limit as $\varepsilon\to0$ the deterministic model is governed by the infinity Laplacian operator \eqref{infty-lapla}.

In \cite{BCF}, the coin flip is replaced by a symmetric $s$-stable stochastic Levy process with $s\in(1/2,1)$, thus leading to a nonlocal integro-differential equation governed by an infinity fractional Laplacian when $\varepsilon\to0$. Choosing appropriately a parameter in this operator leads to the so called {\it H\"older infinity laplacian},
\begin{equation}\label{Linfty}
\mathcal L_s u (x): =\mathcal L_s^+ u (x)+\mathcal L_s^- u (x):= \sup_{y\in\R^N}\frac{u(x)-u(y)}{|x-y|^s} + \inf_{y\in\R^N}\frac{u(x)-u(y)}{|x-y|^s},
\end{equation}
studied in \cite{CLM}. If $g\in C^{0,s}(\partial\Omega)$ the function $u$ solving
 \begin{equation}\label{infty}
\begin{cases}
\mathcal L_s u = 0 & \text{in }\Omega,\\
u = g & \text{on } \partial\Omega,
 \end{cases}
 \end{equation}
produces the optimal  H\"older extension to $\overline\Omega$ of the H\"older boundary data $g$, in the sense that the H\"older seminorm for $u$ in $\Omega$ is  always less than or equal to the one for the boundary data given on $\partial\Omega$. Problem \eqref{infty} is obtained taking limit as $p\to\infty$ in
\begin{equation}\label{cham-chum}
\begin{cases}
(-\Delta_p)^s u = 0 & \text{in }\Omega, \\
u = g & \text{on } \partial \Omega,
\end{cases}
\end{equation}
where $(-\Delta_p)^s$ is the so-called fractional $p-$laplacian defined by
$$
(-\Delta_p)^s u(x) = \text{p.v.} \int_{\R^N} \frac{|u(x)-u(y)|^{p-2}(u(x)-u(y))}{|x-y|^{N+sp}}\, dy.
$$
See \cite{FPL} for the details.

A bridge between fractional order theories and Orlicz-Sobolev settings is provided in \cite{FBS}. In that paper, the authors define the Fractional order Orlicz-Sobolev spaces, associated to the Orlicz function $G$, as
$$
W^{s,G}(\R^N):=\left\{   u\in L^G(\R^N)\colon \Phi_{s,G}(u)<\infty \right\},
$$
being $L^G(\R^N)$, the usual Orlicz-Lebegue space,
$$
L^G(\R^N)=\{u\in L^1_\text{loc}(\R^N)\colon \Phi_{G}(u) < \infty\},
$$
and the modulars $\Phi_G$ and $\Phi_{s, G}$ are determined by
$$
\Phi_{G}(u) := \int_{\R^N} G(u(x))\,dx,\quad \Phi_{s,G}(u) := \iint_{\R^N\times\R^N} G\left( \frac{u(x)-u(y)}{|x-y|^s} \right) \frac{ dx\,dy}{|x-y|^N}.
$$
See the preliminary Section \ref{preliminares} for the most relevant facts about these spaces.

In addition, in \cite{FBS}, the authors recover the classic Orlicz-Sobolev space corresponding to $G$, as $s\to1$, extending the celebrated result by Bourgain, Brezis and Mironescu \cite{BBM} to this fractional Orlicz-Sobolev setting. They conclude obtaining existence and uniqueness results to weak solutions related  to the fractional $g$-laplacian operator, defined as
$$
(-\Delta_g)^s u:=\text{p.v.} \int_{\R^N}  g\left(\frac{u(x)-u(y)}{|x-y|^s}\right)\frac{dy}{|x-y|^{N+s}}
$$
where p.v. stands for {\em in principal value} and $g=G'$. Observe that when $G(t)=|t|^p$, then $g(t)=|t|^{p-2}t$ and hence $(-\Delta_g)^s = (-\Delta_p)^s$ is the fractional $p-$laplacian.

This discussion leads to the main purpose of this paper, the study of the limit problem for
\begin{align} \label{ec.gn}
\begin{cases}
(-\Delta_{g_n})^s u_n= f &\quad \text{ in } \Omega,\\
u_n=0 &\quad \text{ on } \R^N \setminus \Omega,
\end{cases}
\end{align}
where $\Omega\subset \R^N$  is a bounded Lipschitz domain,  $s\in(0,1)$ and $f$ is a suitable given function.

The functions $g_n$ are odd and verify that $g_n(t)=G_n'(t)$, being $\{G_n(t)\}_{n\in\N}$ a sequence of Orlicz functions, satisfying the growth condition
$$
p^-_n\leq \frac{tg_n(t)}{G_n(t)}\leq p^+_n,\quad \text{for any }t>0.
$$

Our main concern in this work is to analyze  the passage to the limit as $n\to\infty$ in the spirit of \cite{FPL}, under the assumption that for some $\beta>1$ it holds
\begin{equation} \label{cond.beta}
p_n^- \to \infty \text{ as } n\to\infty \quad \text{and} \quad p_n^+\leq \beta p_n^- .
\end{equation}
Given that we could take the specific choice $g_n(t) = |t|^{p_n-2}t,$ so that the operator $(-\Delta_{g_n})^s$ agrees with the fractional $p_n-$laplacian, it is natural to expect that the limit problem is the same as the one obtained in \cite{FPL}, as it indeed happens.

The rest of the paper is organized as follows. In Section 2 we introduce some preliminary definitions and properties on fractional Orlicz-Sobolev spaces and  fractional $g-$Laplacian operators. Section 3 is devoted to provide for precise definitions of weak and viscosity solutions of Dirichlet $g-$Laplacian type problems as well as the  conditions under which weak solutions are viscosity ones. In section 4 we establish some a priori estimates ensuring the convergence of sequence of solutions $u_n$ to some function $u_\infty$, which can be explicitly determined when $f$ is positive.


\section{Preliminaries}\label{preliminares}
In this section we make a brief overview on the classical Orlicz-Sobolev spaces, as well as we introduce the Fractional order Orlicz-Sobolev Spaces, their main properties, studied in \cite{FBS}, and the associated fractional $g-$laplacian operator.
\subsection{Orlicz functions}
By an Orlicz function $G\colon\R \to \R$ we understand a function  fulfilling the following properties:
\begin{align*}
\tag{$G_0$}\label{G0} &G \text{ is even, continuous, convex,  increasing for $t>0$ and }  G(0)=0;\\
\tag{$G_1$}\label{G1} &G\text{ satisfies the }\Delta_2 \text{ condition, i.e.}\\& \text{there exists $\C>2$ such that } G(2t)\leq \C G(t), \text{ for any }t>0;\\
\tag{$G_2$}\label{G3} &  \lim_{x\to 0} \frac{G(x)}{x} = 0 \text{ and }\lim_{x\to \infty} \frac{G(x)}{x} = \infty.
\end{align*}
Denoting as $g(t)=G'(t)$ we assume that they are related through the following growth assumption
\begin{equation} \label{cond}
0<p^-\leq \frac{tg(t)}{G(t)} \leq p^+ \quad \forall t>0.
\end{equation}

An immediate consequence of \eqref{cond} is the following polynomial growth, both on $G$ and $g$.
\begin{lema}\label{poliG}
Assume $G$ is an Orlicz function satisfying \eqref{cond} and normalized as $G(1)=1$. Then the following polynomial growth hold
\begin{align}\label{cotaG1}
t^{p^-} &\le G(t) \le t^{p^+} \qquad \text{for } t>1 \\
\label{cotaG2}
t^{p^+} &\le G(t) \le t^{p^-} \qquad \text{for } 0<t<1\\
\label{cotag1}
p^- t^{p^- - 1} &\le g(t) \le p^+ t^{p^+ - 1} \qquad \text{for } t>1 \\
\label{cotag2}
p^- t^{p^+ - 1} &\le g(t) \le p^+ t^{p^- - 1} \qquad \text{for } 0<t<1.
\end{align}
\end{lema}

The normalization condition $G(1)=1$ is by no means restrictive. In fact given $c>0$, if $G$ verifies \eqref{cond}, then $\tilde G(t) = c G(t)$ also satisfies \eqref{cond} with the same constants $p^\pm$. Therefore, choosing $c= G(1)^{-1}$ we conclude
$$
G(1)\min\{t^{p^-}; t^{p^+}\} \le G(t)\le G(1) \max\{t^{p^-}; t^{p^+}\}.
$$
Inequalities \eqref{cotag1} and \eqref{cotag2} are modified accordingly.

This simple observation allows us to prove that
\begin{equation}\label{estimare1}
G(a)\min\{ t^{p^-};  t^{p^+}\} \leq G(a t)\leq G(a)\max\{ t^{p^-};  t^{p^+}\},
\end{equation}
since $\hat G(t) = G(at)$ also fulfils \eqref{cond} with the same constants $p^\pm$.

The next lemma is deduced by combining \eqref{estimare1} with \eqref{cond}.
\begin{lema}
Let $G$ be an Orlicz function satisfying \eqref{cond} and let $\beta>1$ be such that $p^+\le \beta p^-$. Then, for every $t>0$ and $0<t<1$ it holds that
\begin{equation}\label{estimareg}
\beta^{-1} g(a) t^{p^+-1} \le g(a t)\le \beta g(a) t^{p^- - 1}.
\end{equation}
\end{lema}
The proof of this lemma is immediate.

The complementary function of an Orlicz function $G$ is defined as
$$
G^*(a) := \sup\{at-G(t)\colon t>0\}.
$$
From this definition  the following Young-type inequality holds
\begin{equation}\label{young}
at\le G(t) + G^*(a)\quad \text{for every } a,t\ge 0.
\end{equation}

It is easy to deduce the identity,
\begin{equation}\label{igualdad}
G^*(g(t)) = tg(t) - G(t),
\end{equation}
for every $t>0$, see \cite[Lemma 2.9]{FBS}. Now \eqref{igualdad} and \eqref{cond} yield that
\begin{equation}\label{cond*}
(p^+)'\le \frac{t g^*(t)}{G^*(t)}\le (p^-)',
\end{equation}
where $g^*(t) = (G^*)'(t)$. Observe that \eqref{cond*} implies that $G^*$ verifies the $\Delta_2$ condition. See \cite[Theorem 4.1]{Krasnoselskii}.

\begin{rem}\label{reflexivo}
 Notice that indeed, \cite[Theorem 4.1]{Krasnoselskii} entails that \eqref{cond} is equivalent to the fact that $G$ and $G^*$ both satisfy the $\Delta_2$ condition.
\end{rem}

At some point we need to impose a further restriction on $g(t)$, namely
\begin{equation} \label{condlib}
\frac{tg'(t)}{g(t)} \leq p^+-1 \quad \forall t>0.
\end{equation}

It is easy to check that the second inequality in condition \eqref{cond} follows from \eqref{condlib}.

Moreover, conditions \eqref{condlib} and \eqref{estimareg} imply the next lemma.
\begin{lema}
Assume that $G$ verifies \eqref{cond} and $g=G'$ verifies \eqref{condlib}. Then
\begin{equation}\label{estaSI}
g'(ct)\le (p^+-1)\beta \frac{g(c)}{c} t^{p^- - 2},
\end{equation}
for any $c>0$ and $0<t<1$.
\end{lema}

\begin{proof}
Denote $\tilde g(t) = g(ct)$. Then is easy to see that $\tilde g$ also verifies \eqref{condlib}. Therefore, using \eqref{estimareg} we get
$$
\tilde g'(t)\le (p^+-1) \frac{\tilde g(t)}{t} = (p^+-1) \frac{g(ct)}{t}\le (p^+-1) \beta g(c) t^{p^- - 2}.
$$
This concludes the proof.
\end{proof}

\begin{rem}
Throughout this paper it will always be assumed that the Orlicz function $G$ satisfies \eqref{cond}. Whenever \eqref{condlib} is required it will be pointed out explicitly.
\end{rem}

\subsection{Fractional Orlicz--Sobolev spaces}
Given an Orlicz function $G$ and a fractional parameter $0 < s < 1$,  we consider the spaces $L^G(\R^N)$ and $W^{s,G}(\R^N)$ defined as
\begin{align*}
&L^G(\R^N) :=\left\{ u\colon \R^N \to \R \text{ measurable, such that }  \Phi_{G}(u) < \infty \right\},\\
&W^{s,G}(\R^N):=\left\{ u\in L^G(\R^N) \text{ such that } \Phi_{s,G}(u)<\infty \right\},
\end{align*}
where the modulars $\Phi_G$ and $\Phi_{s,G}$ are determined by
\begin{align*}
\Phi_{G}(u)&:=\int_{\R^N} G(u(x))\,dx,\\
\Phi_{s,G}(u)&:= \iint_{\R^N\times\R^N} G\left( \frac{u(x)-u(y)}{|x-y|^s} \right) \frac{ dx\,dy}{|x-y|^N}.
\end{align*}

Observe that,  Remark \ref{reflexivo} and \cite[Theorem 3.13.9 and Remark 3.13.10]{Kufner} guarantee the reflexivity of the space $L^G(\R^N)$. Moreover, by \cite{FBS}, $W^{s,G}(\R^N)$ is also reflexive.

The following definitions will simplify the notation. Given $u\in L^1_\text{loc}(\R^N)$  we define
\begin{equation}\label{Ds}
D^s u(x,y) := \frac{u(x)-u(y)}{|x-y|^s},
\end{equation}
the H\"older quotient of order $s$.

Let us also denote by $\mu$ the measure on $\R^N\times\R^N$ given by
\begin{equation}\label{mu}
d\mu := \frac{dxdy}{|x-y|^N}.
\end{equation}
 Note that $\mu$ is a regular Borel measure, though not a Radon measure, since any open set containing points on the diagonal $\Delta := \{(x,x)\colon x\in\R^N\}$ has infinite $\mu-$measure.

With these notations at hand, the fractional Orlicz-Sobolev space can be expressed as
$$
W^{s,G}(\R^N)=\{u\in L^G(\R^N) \colon D^s u\in L^G(\R^N\times\R^N, d\mu)\}
$$
while the associated modular is written as
$$
\Phi_{s,G}(u) = \iint_{\R^N\times\R^N} G(D^s u)\, d\mu.
$$

These spaces are endowed with the so-called Luxemburg norms
$$
\|u\|_G =  \|u\|_{L^G(\R^N)} := \inf\left\{\lambda>0\colon \Phi_G\left(\frac{u}{\lambda}\right)\le 1\right\}
$$
and
$$
\|u\|_{s,G} =  \|u\|_{W^{s,G}(\R^N)} := \|u\|_G + [u]_{s,G},
$$
where
$$
[u]_{s,G} :=\inf\left\{\lambda>0\colon \Phi_{s,G}\left(\frac{u}{\lambda}\right)\le 1\right\}.
$$
For $0<s<1$, the term $[\, \cdot\,]_{s,G}$ will be called the {\em $(s,G)-$Gagliardo seminorm}.

Young's inequality \eqref{young} easily implies the following H\"older-type inequality in Orlicz spaces
\begin{lema}
Let $G$ be an Orlicz function and $G^*$ its complementary function. Then for every $u\in L^G(\Omega)$ and $v\in L^{G^*}(\Omega)$, it holds
\begin{equation}\label{HolderG}
\int_\Omega |uv|\, dx\le 2 \|u\|_G \|v\|_{G^*}.
\end{equation}
\end{lema}

\begin{proof}
The proof can be found in any of the above mentioned references on Orlicz spaces, for instance in \cite{Krasnoselskii}. We include the proof for completeness.

Assume first that $\|u\|_G = \|v\|_{G^*}=1$, then $\Phi_G(u) = \Phi_{G^*}(v)=1$ and hence, using \eqref{young},
\begin{equation}\label{casiH}
\int_\Omega |uv|\, dx \le \Phi_G(u) + \Phi_{G^*}(v) = 2.
\end{equation}

Now, the proof of \eqref{HolderG} follows by considering $\bar u = u/\|u\|_G$ and $\bar v= v/\|v\|_{G^*}$ in \eqref{casiH}.
\end{proof}

We also introduce the space,
$$
W^{s,G}_0(\Omega) := \left\{u\in W^{s,G}(\R^N) \colon u=0 \text{ a.e. in } \R^N\setminus \Omega \right\}.
$$

Alternatively, one can consider
$$
\widetilde{W}^{s,G}(\Omega) := \overline{C_c^\infty(\Omega)}^{\|\cdot\|_{s,G}}.
$$

In the classical case, i.e. when $G(t)=t^p$, these spaces $W^{s,p}_0(\Omega)$ and $\widetilde{W}^{s,p}(\Omega)$ are known to coincide when $s<\tfrac{1}{p}$ or if $0<s<1$ and $\Omega$ has Lipschitz continuous boundary, see \cite{Kufner}.

Next lemma deals with the inclusion of Orlicz-Lebesgue spaces into the usual Lebesgue spaces.

\begin{lema} \label{lema.inclusion}
Let $G$ be an Orlicz function verifying \eqref{cond} and let $r\in [p^+,\infty]$. Then, there holds
$$
L^r(\Omega)\subset L^G(\Omega).
$$
Furthermore,
$$
\|u\|_G\le \max\{1; |\Omega|^{\frac{1}{p^-} - \frac{1}{r}} + |\Omega|^{\frac{1}{p^-} - \frac{p^+}{p^-}\frac{1}{r}}\} \|u\|_r.
$$
\end{lema}

\begin{proof}
A proof of this fact can be found, for instance, in \cite[Theorem 3.17.1]{Kufner}. However, we  include the proof taking care of  the explicit dependence of the constants on $G$, which will be necessary in forthcoming arguments.

Assume first that $p^+\le r<\infty$.   We decompose $\Phi_G(u)$ as follows.
$$
\int_\Omega G(u)\, dx = \left(\int_{|u|\ge 1} + \int_{|u|<1}\right) G(u)\, dx = I + II.
$$
Assume that $\|u\|_r=1$. An estimate on $I$ follows from  \eqref{cotaG1}. Precisely,
$$
I\le \int_{\Omega} |u|^{p^+}\, dx \le \|u\|_r^{p^+} |\Omega|^{1-\frac{p^+}{r}} = |\Omega|^{1-\frac{p^+}{r}}.
$$
On the other hand, using now \eqref{cotaG2}  we obtain
$$
II\le |\Omega|^{1-\frac{p^-}{r}}.
$$
Now, as long as $\|u\|_G\ge 1$, we have that
$$
\|u\|_G\le \left(\int_\Omega G(u)\, dx\right)^\frac{1}{p^-},
$$
and this concludes the case $\|u\|_r=1$. 

If $\|u\|_r\neq 1$ the result follows using the homogeneity of the norm $\|\cdot\|_G$.

The case $r=\infty$ is analogous and is left to the reader.
\end{proof}

Furthermore, there holds an embedding result of fractional Orlicz-Sobolev spaces into the usual fractional Sobolev ones. We need first the following Lemma.
\begin{lema}\label{omega=Rn}
Let $\Omega\subset\R^N$ be a bounded Lipschitz domain. Let $0<t<1<q<\infty$ be such that $t\neq \frac{1}{q}$. Then,
$$
\iint_{\R^N\times\R^N} |D^t u|^q\, d\mu\le C \iint_{\Omega\times\Omega} |D^t u|^q\, d\mu\
$$
for every $u\in C^\infty_c(\Omega)$  and some constant $C=C(N,t,q,\Omega)$.
\end{lema}

\begin{proof}
This is the content of \cite[Corollary 1.4.4.5]{grisvard}.
\end{proof}

\begin{prop} \label{prop.reg}
Let $G$ be an Orlicz function satisfying condition \eqref{cond}. Then, given $1\le q<p^-$ and $0<t<s$ it holds that
$$
[u]_{t,q;\Omega} := \left(\iint_{\Omega\times\Omega}|D^t u|^q\, d\mu\ \right)^\frac{1}{q}\le C(\Omega, N, q, s-t) [u]_{s,G},
$$
for every $u\in W^{s,G}_0(\Omega)$.
\end{prop}

\begin{proof}
The proof follows the ideas of \cite[Section 3.17]{Kufner}. In fact, from \eqref{cotaG1} it is inferred that,
\begin{equation}\label{cotaG}
a^q\le a^{p^-} \le G(a), \quad\text{for every } a\ge 1.
\end{equation}
Now, let $u\in W^{s,G}_0(\Omega)$ and define
\begin{align*}
& A := \{(x,y)\subset \Omega\times\Omega\colon |D^s u(x,y)|\le 1\},\\
& B := \Omega\times\Omega\setminus A.
\end{align*}

We compute,
\begin{align*}
\iint_{\Omega\times\Omega} |D^tu|^q&\, d\mu = \iint_{\Omega\times\Omega} |D^su|^q |x-y|^{(s-t)q}\, d\mu\\&=\left(\iint_A + \iint_B\right)|D^su(x,y)|^q \frac{1}{|x-y|^{N-(s-t)q}}\, dxdy\\
&= I + II.
\end{align*}
Notice that
$$
I\le \iint_{\Omega\times\Omega} \frac{1}{|x-y|^{N-(s-t)q}}\, dxdy \le |\Omega| N\omega_N \frac{\d^{q(s-t)}}{q(s-t)},
$$
where $\d = \d(\Omega)$ is the diameter of $\Omega$.

To estimate the second term, we invoke \eqref{cotaG} and obtain
$$
II \le \d^{q(s-t)} \Phi_{s,G}(u)<\infty.
$$

Therefore, as long as $[u]_{s,G} = 1$, then $\Phi_{s,G}(u)=1$ and hence
\begin{equation}\label{constante.explicita}
[u]_{t,q;\Omega}\le \left(\frac{|\Omega| N\omega_N}{q(s-t)} + 1\right)^\frac{1}{q} \d^{s-t}.
\end{equation}
From this inequality, the proof concludes from homogeneity of the seminorm.
\end{proof}

\begin{cor}\label{cor.inmersion}
Under the same assumptions and notations of the previous proposition, it holds that
$$
W^{s,G}_0(\Omega)\subset W^{t,q}_0(\Omega),
$$
with continuous inclusion.
\end{cor}

\begin{proof}
The proof is immediate from Proposition \ref{prop.reg} and Lemma \ref{omega=Rn}
\end{proof}

\begin{rem}\label{remark.clave}
In the course of the proofs,  it will be necessary to understand the asymptotic behavior of the constants. We state this behavior for future reference. Inequality \eqref{constante.explicita} reveals that the constant $C(\Omega, N, q, s-t)$ in Proposition \ref{prop.reg} behaves as
$$
\limsup_{q\to\infty} C(\Omega, N, q, s-t) \le \d^{s-t},
$$
and hence,
$$
\limsup_{t\to s} \limsup_{q\to\infty} C(\Omega, N, q, s-t)\le 1.
$$
\end{rem}

We finish this section showing some Poincar\'e type inequalities. Next theorem states that $[\cdot]_{s,G}$ is an equivalent norm to $\|\cdot\|_{s,G}$ in $W^{s,G}_0(\Omega)$.
\begin{thm}\label{thm.poincare}
Let $\Omega\subset \R^N$ be open and bounded. Then,
$$
\Phi_G(u) \le  \Phi_{s,G}(C \d^s u),
$$
for every $0<s<1$ and $u\in W^{s,G}_0(\Omega)$, where $\d=\d(\Omega)$ stands for the diameter of $\Omega$ and $C = \left(\frac{sp^+}{N\omega_N}\right)^\frac{1}{p^-}$.
\end{thm}

\begin{proof}
The proof is similar to that contained in \cite[Corollary 6.2]{FBS}.

Indeed, observe that whenever $x\in \Omega$ and $|x-y|\ge \d$, then $y\not\in\Omega$. Thus,
\begin{align*}
\Phi_{s,G}(C \d^s u) &= \iint_{\R^N\times\R^N} G(C\d^s D^s u(x,y))\, d\mu(x,y)\\
& \ge \int_\Omega \int_{|x-y|\ge \d} G\left(\frac{C\d^s}{|x-y|^s} u(x)\right)\, \frac{dxdy}{|x-y|^N}.
\end{align*}
Without loss of generality, we can assume that $C\ge 1$. Invoking condition \eqref{estimare1} we get
$$
\Phi_{s,G}(C \d^s u)\ge C^{p^-} \d^{sp^+} \left(\int_\Omega G(u(x))\, dx\right) \left(\int_{|z|\ge \d} \frac{dz}{|z|^{N+sp^+}}\right).
$$
At this point the result follows just observing that
$$
\int_{|z|\ge \d} \frac{dz}{|z|^{N+sp^+}} = \frac{N\omega_N}{sp^+} \d^{-sp^+},
$$
and choosing $C$ appropriately.
\end{proof}

As a corollary we infer the following Poincar\'e's  inequality for fractional Luxemburg type norms.
\begin{cor} \label{poincare.norma}
Let $\Omega\subset \R^N$ be open and bounded. Then
$$
\|u\|_{G} \leq C \d^s [u]_{s,G}
$$
for every $0<s<1$ and $u\in W^{s,G}_0(\Omega)$, where $C$ depends on $N, s, p^+$ and $p^-$.

\end{cor}
\begin{proof}
Given $u\in W^{s,G}_0(\Omega)$ applying Theorem \ref{thm.poincare} to the function $\frac{u}{C\d^s [u]_{s,G}}$,  we get
$$
\Phi_G\left(\frac{u}{C\d^s [u]_{s,G}}\right) \leq \Phi_{s,G}\left( \frac{u}{[u]_{s,G}}\right) =1
$$
by definition of the Luxemburg norm. Consequently,
$$
\|u\|_G = \inf\{\lam\colon \Phi_G\left(\frac{u}{\lam}\right)\leq 1\} \leq C \d^s [u]_{s,G}
$$
as desired.
\end{proof}

\subsection{The fractional $g-$laplacian operator}

Let $G$ be an Orlicz function and $0<s<1$  a fractional parameter. In order to introduce our fractional operators, we first need to define the {\em fractional divergence}.

To this end, consider $\phi = \phi(x,y)$ a measurable function defined on $\R^N\times\R^N$ and denote as its $s-$divergence the integral in the sense of {\em principal value}
\begin{equation}\label{sdiv}\begin{split}
\diver^s\phi (x) &= \text{p.v.} \int_{\R^N} (\phi(y,x)-\phi(x,y))\, \frac{dy}{|x-y|^{N+s}} \\
&= \lim_{\ve\downarrow 0} \int_{|h|>\ve} (\phi(x+h,x)-\phi(x,x+h))\, \frac{dh}{|h|^{N+s}}.
\end{split}
\end{equation}
We first need to make precise the sense in which the limit in \eqref{sdiv} exists.

For $\ve>0$, let us denote
$$
\diver^s_\ve\phi (x) := \int_{|h|>\ve} (\phi(x+h,x)-\phi(x,x+h))\, \frac{dh}{|h|^{N+s}}.
$$

First, a preliminary lemma.
\begin{lema}\label{divep}
Let $G$ be an Orlicz function and $\phi=\phi(x,y)\in L^G(\R^N\times\R^N, d\mu)$. Then $\diver^s_\ve \phi\in L^G(\R^N)$. Moreover,
$$
\int_{\R^N} G(\diver^s_\ve \phi)\, dx \le C_\ve \iint_{\R^N\times\R^N} G(\phi)\, d\mu.
$$
\end{lema}

\begin{proof}
The proof is a consequence of Jensen's inequality. In fact,
\begin{align*}
G\left(\int_{|h|>\ve} \phi(x,x+h) |h|^{-N-s}\, dh\right) &\le C_\ve \int_{|h|>\ve} G(\phi(x,x+h)) |h|^{-N-s}\, dh\\
&\le \frac{C_\ve}{\ve^s}\int_{|h|>\ve} G(\phi(x,x+h)) |h|^{-N}\, dh.
\end{align*}
Integrating this estimate on $\R^N$, together with the $\Delta_2$ condition concludes the proof.
\end{proof}

Now we are ready to prove that the fractional divergence $\diver^s$, is well defined.
\begin{prop}\label{divers}
Let $G$ be an Orlicz function and $G^*$ be its complementary function.  Then,
$$
\diver^s\colon L^{G^*}(\R^N\times\R^N,d\mu)\to W^{-s, G^*}(\R^N),
$$
where $W^{-s,G^*}(\R^N)$ stands for the (topological) dual space of $W^{s,G}(\R^N)$. Furthermore, $\diver^s$ is bounded and the integration by parts formula holds, i.e.
\begin{equation}\label{partes}
\langle \diver^s\phi, u\rangle = -\iint_{\R^N\times\R^N} \phi D^s u\, d\mu,
\end{equation}
for every $\phi\in L^{G^*}(\R^N\times\R^N, d\mu)$ and every $u\in W^{s,G}(\R^N)$.
\end{prop}

\begin{proof}
First, observe that, by Lemma \ref{divep},
$$
\langle \diver^s\phi, u\rangle = \lim_{\ve\to0}\langle \diver_\ve^s\phi, u\rangle =\lim_{\ve\to0} \int_{\R^N} \diver_\ve^s \phi(x) u(x)\, dx.
$$
Moreover,
\begin{align*}
\int_{\R^N} \diver_\ve^s \phi(x) u(x)\, dx &= \int_{\R^N}\int_{|x-y|>\ve} (\phi(y,x)-\phi(x,y)) u(x)\frac{dxdy}{|x-y|^{N+s}}\\
& = \int_{\R^N}\int_{|x-y|>\ve} (\phi(x,y)-\phi(y,x)) u(y)\frac{dxdy}{|x-y|^{N+s}}.
\end{align*}
Therefore
$$
\int_{\R^N} \diver^s_\ve\phi(x) u(x)\, dx = -\int_{\R^N}\int_{|x-y|>\ve} \phi(x,y) D^s u(x,y) d\mu(x,y).
$$
Since $\phi D^s u\in L^1(\R^N\times\R^N, d\mu)$,  the dominated convergence theorem enables to pass to the limit as $\ve\downarrow 0$ and the result follows.
\end{proof}

\begin{rem}
According to these notations, the fractional Laplacian, up to some normalization constant, can be written as
$$
(-\Delta)^s u(x) = -\diver^s (D^s u) (x).
$$
Moreover, in these terms the fractional $p-$Laplacian reads as
\begin{align*}
(-\Delta_p)^s u(x) &:= 2 \text{p.v}. \int_{\R^N} \frac{|u(x)-u(y)|^{p-2}(u(x)-u(y))}{|x-y|^{N+sp}}\, dy\\
& = -\diver^s(|D^s u|^{p-2}D^s u) (x).
\end{align*}
\end{rem}

Now, let $G$ be an Orlicz function and denote $g=G'$. We define the fractional $g-$laplacian operator, $(-\Delta_g)^s$ as
\begin{align} \label{vp}
(-\Delta_g)^s u(x):&= -\diver^s(g(D^s u))(x) \\
&= 2 \text{p.v.} \intr  g\left(\frac{u(x)-u(y)}{|x-y|^s}\right)\frac{dy}{|x-y|^{N+s}}.
\end{align}

Observe that, in view of Proposition \ref{divers}, the operator
$$
(-\Delta_g)^s\colon W^{s,G}(\R^N)\to W^{-s,G^*}(\R^N)
$$
is  continuous. In addition, by the integration by parts formula \eqref{partes},
$$
\langle (-\Delta_g)^s u,v \rangle =  \iint_{\R^N\times\R^N} g(D^s u)  D^s v\, d\mu,
$$
for any $v\in W^{s,G}(\R^N)$.

In order to introduce in the next Section the concept of viscosity solutions, it is necessary to give a point-wise sense to our operator. This sense will be ensured according to the value of the parameter $p^-$.

Specifically, if $u\in C^1(\R^N)\cap L^\infty(\R^N)$, then we have to require that $p^->\frac{1}{1-s}$, which could be large if $s$ is close to one.  However, since in this work the values $p^-_n$ diverge eventually to infinity, this assumption is not restrictive and it is enough for our purposes.

Nevertheless, we also show that  whenever $u\in C^2(\R^N)\cap L^\infty(\R^N)$, then it suffices with $g'$ locally bounded to ensure that  \eqref{vp} holds for every $x\in\R^N$.  This fact is guaranteed, for instance, by \eqref{condlib} and the assumption $p^->2$, which is a condition independent of $s$, though based on requiring more regularity on $u$.

This is the core of the following lemma.

\begin{lema}\label{pointwise}
Let $0<s<1$. Assume either
\begin{itemize}
\item $p^->\frac{1}{1-s}$ and $u\in C^1(\R^N)\cap L^\infty(\R^N)$ or
\item $g'$ is locally bounded and $u\in C^2(\R^N)\cap L^\infty(\R^N)$.
\end{itemize}
Then  \eqref{vp} is well defined pointwise for every $x\in\R^N$.
\end{lema}

\begin{proof}
We decompose definition \eqref{vp} into
\begin{align*}
(-\Delta_g)^s u :=& \int_{|x-y|>1} g(D^s u(x,y))\frac{dy}{|x-y|^{N+s}}\\
&+ \lim_{\varepsilon\to0}\int_{\ve<|x-y|\le 1}  g(D^s u(x,y))\frac{dy}{|x-y|^{N+s}}\\
= & \mathcal I_1+\lim_{\varepsilon\to0}\mathcal I_2^\ve.
\end{align*}
To bound $\mathcal I_1$, observe that
$$
|D^s u(x,y)|\le 2\|u\|_\infty,
$$
whenever $|x-y|>1$. Taking into account that $g$ is nondecreasing, we obtain
$$
\mathcal I_1 \le g(2\|u\|_\infty) \int_{|x-y|>1} |x-y|^{-N-s}dy = \frac{N\omega_N}{s} g(2\|u\|_\infty),
$$
where $\omega_N$ denotes the measure of the unit ball in $\R^N$.

Suppose first that $u\in C^1(\R^N)\cap L^\infty(\R^N)$. It holds that
\begin{equation}\label{cotaDs}
|D^s u(x,y)|\le L |x-y|^{1-s},
\end{equation}
where $L = \|\nabla u\|_{L^\infty(B_1(x))}$. Hence, as long as $|x-y|\le 1$, with the use of \eqref{cotaDs} and \eqref{estimareg} we can proceed with $\mathcal I_2^\ve$ as follows,
\begin{align*}
\int_{\ve<|x-y|\le 1}  g(D^s u(x,y))\frac{dy}{|x-y|^{N+s}}&\leq \int_{\ve<|x-y|\le 1} \frac{g(L|x-y|^{1-s})}{|x-y|^{N+s}}\,dy \\
&= N \omega_N  \int_0^1 \frac{g\left(Lr^{1-s}\right)}{r^{s+1}}\,dy\\
&\leq N \omega_N \beta g(L) \int_0^1 \frac{r^{(1-s)(p^- -1)}}{r^{s+1}}\,dy\\
&\leq N \omega_N \beta g(L)\int_0^1 r^{(1-s)p^- -2}\,dy.
\end{align*}

The last quantity is finite whenever $p^->\frac{1}{1-s}$. The first statement is now completed.

To show the second statement it just remains to see the boundedness of $\mathcal I_2^\ve$ assuming that $u\in C^2(\R^N)\cap L^\infty(\R^N)$ and $g'$ is locally bounded. Notice that
$$
\mathcal I_2^\ve = \int_{\ve<|z|\le 1} g(D^s u (x,x+z))\frac{dz}{|z|^{N+s}} =  \int_{\ve<|z|\le 1} g(D^s u(x,x-z))\frac{dz}{|z|^{N+s}}.
$$
Therefore, we get
\begin{equation}\label{eq.intermedia}
\mathcal I_2^\ve = \frac12 \int_{\ve<|z|\le 1} \left(g(D^s u (x,x+z)) - g(D^s u(x-z,x))\right) \frac{dz}{|z|^{N+s}}.
\end{equation}
According to the function
$$
\varphi(t) := g(tD^s u(x,x+z) + (1-t) D^s u(x-z,x)),
$$
equation \eqref{eq.intermedia} reads as
\begin{equation}\label{I2ve}
\mathcal I_2^\ve=\frac12\int_{\ve<|z|\le 1} \frac{\varphi(1) - \varphi(0)}{|z|^{N+s}}\, dz = \frac12\int_{\ve<|z|\le 1}\int_0^1 \varphi'(t)\, dt\, |z|^{-(N+s)}\, dz.
\end{equation}
Recall that
$$
\varphi'(t) = g'(tD^s u(x,x+z) + (1-t) D^s u(x-z,x)) \frac{D^2_z u(x)}{|z|^s},
$$
where
$$
D^2_z u(x) =- (u(x+z) - 2u(x) + u(x-z)).
$$
Taking into account that  $g'$ is locally bounded and  \eqref{cotaDs}, we deduce, for $|z|\le 1$ and $0\le t\le 1$, that
$$
|g'(tD^s u(x,x+z) + (1-t) D^s u(x-z,x))|\le C.
$$
The fact that $u\in C^2(\R^N)$ implies that
$$
|D^2_z u(x)|\le C |z|^2,
$$
for $|z|\le 1$.

Assembling  all these bounds together, we conclude that the integrand in \eqref{I2ve} is bounded by $C |z|^{-(N+2s-2)}\in L^1(B_1)$. Hence the limit exists and the proof of the second statement follows.
\end{proof}


\section{Weak and viscosity solutions}
We devote this section to specify the different concepts of solutions that will be treated throughout this work. Weak solutions will be considered for a fixed value of $n\in\mathbb N$. However, since the objective is the asymptotic analysis of those solutions as $n\to\infty$, we need to introduce the notion of viscosity solutions. Actually, for $n$ sufficiently large, weak solutions are indeed viscosity solutions.

For the sake of simplicity we drop off the subscript $n$ along this section, as long as the definitions are given for $n$ fixed.

Given a bounded open set $\Omega\subset \R^N$ and $f\in L^{G^*}(\Omega)$, we first  provide a notion of weak solution for the following Dirichlet type equation
\begin{align} \label{ec.g}
\begin{cases}
(-\Delta_g)^s u= f &\quad \text{ in } \Omega\\
u=0 &\quad \text{ on } \R^N \setminus \Omega.
\end{cases}
\end{align}

\begin{defn}
We say that $u\in W^{s,G}_0(\Omega)$ is a weak supersolution (subsolution) to \eqref{ec.g} if
\begin{equation} \label{debil}
\langle (-\Delta_g)^s u, v\rangle \geq (\leq) \int_\Omega fv \,dx
\end{equation}
for all nonnegative $v\in W^{s,G}_0(\Omega)$.

If $u$ is a weak super and subsolution, then we say that $u$ is a weak solution to \eqref{ec.g}.
\end{defn}

\begin{rem}
Observe that  $C^\infty_c(\Omega)\subset W^{s,G}_0(\Omega)$, therefore a weak solution of \eqref{ec.g} is a solution in the sense of distributions.
\end{rem}

We are ready now to establish the notion of viscosity solutions to our problem. To this end we will assume that $p^->1/(1-s)$ in order to have the operator $(-\Delta_g)^s$ well defined pointwise for test functions (see Lemma \ref{pointwise}).
\begin{defn} An upper semicontinuous function $u$ such that $u\le 0$ in $\R^N\setminus\Omega$ is a  \emph{viscosity subsolution} to \eqref{ec.g} if
whenever $x_0\in\Omega $ and $\varphi\in C^1_c(\R^N)$
are such that
\begin{itemize}
\item[i)] $\varphi(x_0)=u(x_0)$
\item[ii)] $u(x)\le \varphi(x)$ for $x\ne x_0$
\end{itemize}
 then
 $$
 (-\Delta_g)^s \varphi(x_0)\leq f(x_0).
 $$

\end{defn}
\begin{defn} A lower semicontinuous function $u$ such that $u\ge 0$ in $\R^N\setminus\Omega$ is a  \emph{viscosity supersolution} to \eqref{ec.g} if whenever $x_0\in\Omega $ and $\phi\in C^1_c(\R^N)$
are such that
\begin{itemize}
\item[i)] $u(x_0)=\phi(x_0)$
\item[ii)] $u(x)\ge \phi(x)$ for all $x\neq x_0$,
\end{itemize}
 then
 $$
(-\Delta_g)^s   \phi(x_0)\geq f(x_0).
 $$
\end{defn}
\begin{defn} A continuous function $u$  is a viscosity solution to \eqref{ec.g} if it is a viscosity supersolution and a viscosity subsolution.
\end{defn}

\begin{rem}\label{rem.cte} Mind that $(-\Delta_g)^s  (\phi+C)=(-\Delta_g)^s \phi$ hence the previous definitions are equivalent if  the function $\phi(x)+C$ (or $\phi(x)-C$) touches $u$ from below  (from above, respectively) at $x_0$.

Furthermore, in the previous definitions we may assume that the test function touches $u$ strictly. Indeed, for a test function $\phi$ touching $u$ from below, consider the function $h(x)=\phi(x)- \eta(x)$, where $\eta\in C^\infty_c(\R^N)$ satisfies $\eta(x_0)=0$ and $\eta(x)>0$ for $x\ne x_0$. Notice that $h$ touches $u$ strictly. Moreover, since the function $g$ is increasing it holds that $(-\Delta_g)^s h(x_0)\ge (-\Delta_g)^s \phi(x_0)$.
\end{rem}

For further details about general theory of viscosity solutions we refer to \cite{CIL}, and \cite{J}, \cite{JLM2} for viscosity solutions related to the (local) $\infty-$Laplacian and the $p-$Laplacian  operators.  Regarding  the approach of viscosity solutions  to  equations related to the fractional $p$-Laplacian, see for instance \cite{CLM},\cite{FPL} and \cite{LL}.

Our goal now is to prove that  continuous weak solutions to \eqref{ec.g} are also viscosity solutions to \eqref{ec.g}. We follow the approach given in \cite{FPL}, see also \cite{LL}.

\begin{lema}\label{debilesvisco} Let $G$ be an Orlicz function satisfying \eqref{condlib}. Moreover, assume that $p^->\frac{1}{1-s}$. Then, if $u\in W^{s,G}_0(\Omega)$  is a continuous weak solution to \eqref{ec.g} then, $u$ is also a viscosity solution.
\end{lema}

\begin{proof} We first prove that if $u$ is a continuous weak supersolution then, it is a viscosity supersolution. Arguing \textit{ad contrarium},  suppose that this is not the case. In other words, admit that there exists $\phi$ and  $x_0 \in \Omega$  such that $u(x_0)=\phi (x_0)$,  $u-\phi$ has a strict minimum at $x_0$, and $\phi$ verifies
$$
 (-\Delta_g)^s\phi(x_0)<f(x_0).
$$
Continuity ensures the existence of a radius $r>0$ for which indeed
$$
(-\Delta_g)^s \phi(x)<f(x)\quad \text{for all } x\in B(x_0,r) \subset \Omega.
$$

The test function is perturbed as follows. Declare $\Phi_\ve (x) = \phi(x) + \ve h(x)$ being $0<\ve<1$, $h\in C^1_c(\R^N)$ such that $h(x_0)>0$ and $h\equiv 0$ in $\R^N\setminus B(x_0,r)$.

The idea is now to see that the operator applied on this perturbation remains close in a neighbourhood of $x_0$  if $\ve$ is small, namely
\begin{equation}\label{eq.uniforme}
|(-\Delta_g)^s \Phi_\ve(x)-(-\Delta_g)^s \phi(x)|\le C\varepsilon\quad \text{for all } x\in B(x_0,r).
\end{equation}
In particular, for $\ve$ small
$$
(-\Delta_g)^s \Phi_\ve(x)<f(x)\quad \text{for all } x\in B(x_0,r).
$$

If we multiply by a non-negative continuous function $\psi\in W^{s,G}_0(\Omega)$ supported in $B(x_{0},r)$ and integrate, using the integration by parts formula \eqref{partes}, it yields
$$
\iint_{\R^N\times\R^N} g(D^s \Phi_\ve) D^s \psi \, d\mu < \int_{\R^N} f\psi \, dx.
$$

Having in mind that  $u$ is a weak supersolution, it gives
$$
\iint_{\R^N\times\R^N} g(D^s \Phi_\ve) D^s \psi \, d\mu< \iint_{\R^N\times\R^N} g(D^s u) D^s \psi \, d\mu.
$$
We recall that $\Phi_\varepsilon\le u$ in $\R^N\setminus B(x_0,r)$ and invoke the comparison principle in Lemma~\ref{compa} below, to conclude that $\Phi_\varepsilon\leq u$ in $B(x_0,r)$. But this contradicts that $\Phi_\varepsilon(x_0)>u(x_0)$. This shows  that $u$ is a viscosity supersolution.  The proof of the fact that $u$ is a viscosity subsolution runs similarly.

It remains to show the claim in \eqref{eq.uniforme}. For this purpose, we need to estimate the term
$$
|g(D^s \Phi_\ve) -g(D^s\phi))|.
$$
Notice that
\begin{align*}
g(D^s \Phi_\ve) -g(D^s\phi) &= \int_0^1 \frac{d}{dt}\left(g(D^s \phi - t(D^s \phi - D^s \Phi_\ve))\right)\, dt\\
&= \ve D^s h \int_0^1 g'(D^s \phi + t\ve D^s h)\, dt.
\end{align*}
As a result,
\begin{align*}
|(-\Delta_g)^s \Phi_\ve(x)&-(-\Delta_g)^s \phi(x)| \le 2\int_{\R^N} |g(D^s\Phi_\ve(x,y)) - g(D^s\phi(x,y))|\, \frac{dy}{|x-y|^{N+s}}\\
&\le 2\ve \int_{\R^N} |D^sh(x,y)| \int_0^1 |g'(D^s\phi(x,y) + t\ve D^s h(x,y))|\, dt\, \frac{dy}{|x-y|^{N+s}}.
\end{align*}
Hence \eqref{eq.uniforme} will be proved if we show that the last integral is bounded uniformly in $x\in B_r(x_0)$.

As usual, we split the integral in $|x-y|\le 1$ and $|x-y|>1$. Away from the origin, we have that $\|D^s\phi\|_\infty\le 2\|\phi\|_\infty$ and $\|D^s h\|_\infty\le 2\|h\|_\infty$. Therefore, since $g'$ is locally bounded we conclude that
\begin{align*}
\int_{|x-y|\ge 1} |D^s h(x,y)| &\int_0^1 |g'( D^s \phi(x,y) +t\ve D^s h(x,y))|\,dt\,\frac{dy}{|x-y|^{N+s}} \\
&\le C \int_{|x-y|\ge 1} \frac{dy}{|y-x|^{N+s}}\leq C.
\end{align*}
While close to zero, the Lipschitz continuity of $\phi$ and $h$, estimate \eqref{estaSI} and the fact that $g'$ is increasing, lead to
\begin{align*}
\int_{|x-y|<1} &|D^s h(x,y)| \int_0^1 g'(D^s\phi(x,y) + t\ve D^s h(x,y)) \,dt\, \frac{dy}{|x-y|^{N+s}} \\
&\le C \int_{|x-y|<1} \frac{1}{|x-y|^{2s+N-1}}g'(C|x-y|^{1-s})\,dy \\
&\le C \int_{|x-y|<1} |x-y|^{1-2s-N+(1-s)(p^--2)}dy\\
&= C \int_{|x-y|<1} |x-y|^{-1-N+(1-s)p^-}dy\leq C,
\end{align*}
whenever $p^->1/(1-s)$. The constant $C$ varies from line to line and depends on the Lipschitz constants for $h$ and $\phi$ and on $p^+$. The estimate \eqref{eq.uniforme} is proved.
\end{proof}

\begin{rem}
Observe that if, in addition $p^->\frac{N}{s}$, then we can take $q>1$ and $t<s$ such that $p^->q>\frac{N}{t}>\frac{N}{s}$. Now Proposition~\ref{prop.reg} ensures that $W_0^{s,G}(\Omega)\subset W_0^{t,q}(\Omega)\subset C^{0,\alpha}(\Omega)$. Accordingly, we can remove the continuity assumption in the previous Lemma.
\end{rem}

\begin{rem}
 We note that there is no need to require condition \eqref{condlib} except for the proof of Lemma \ref{debilesvisco}. Moreover, this condition is only used in the proof of \eqref{eq.uniforme}.
\end{rem}

It remains to show the Comparison Principle invoked in Lemma~\ref{debilesvisco}. The ideas in \cite{LL} apply to this more general setting. We include the proof for convenience of the reader.

\begin{lema}\label{compa}Let $\Omega\subset \R^N$ be a domain and  $u,v\in  W^{s,G}(\Omega)$  be two continuous functions satisfying
\begin{itemize}
\item[i)] $v\geq u$ in $\R^N\setminus \Omega$;
\item[ii)] for any non-negative continuous function $\psi \in W^{s,G}_0(\Omega)$
$$
\iint_{\R^N\times \R^N}g(D^s v) D^s\psi \, d\mu \geq \iint_{\R^N\times \R^N} g(D^s u) D^s\psi\, d\mu.
$$
\end{itemize}
Then  $v\geq u$ in $\R^N$.
\end{lema}

\begin{proof} We rewrite hypothesis ii) above as follows
\begin{equation}\label{desigualdad.integral}
\begin{split}
0&\leq \iint_{\R^N\times \R^N}\left(g(D^s v) - g(D^s u)\right) D^s\psi\, d\mu\\
&= \iint_{\R^N\times \R^N}(D^s v - D^s u)\int_0^1 g'(D^s v - t(D^s v - D^s u))\, dt D^s \psi\, d\mu.
\end{split}
\end{equation}
Denote by $w=v-u$ and observe that $w\ge 0$ in $\R^N\setminus \Omega$ by hypothesis i). On the other hand,
\begin{equation}\label{Dsw-}
\begin{split}
D^s w(x,y) D^s w^-(x,y) &= -\frac{w(x)w^-(y) + w(y) w^-(x)}{|x-y|^s} - (D^s w^-(x,y))^2\\
&\le -(D^s w^-(x,y))^2\le 0.
\end{split}
\end{equation}
Moreover, from \eqref{Dsw-}, the fact that $D^sw D^s w^- = 0$, implies that $D^s w^- = 0$ and hence $w^-= const$. The boundary condition $w^- = 0$ in $\R^N\setminus \Omega$ leads to $w^-=0$.

Now, take $\psi= w^-$ in \eqref{desigualdad.integral} and obtain
$$
0\le \iint_{\R^N\times\R^N} D^s w D^s w^- \int_0^1 g'(D^s v - t(D^s v - D^s u))\, dt\, d\mu.
$$
Recall that $g'\ge 0$ and $g'(t)=0$ if and only if $t=0$. Then, from \eqref{Dsw-} we infer that $D^s w D^s w^- = 0$ $\mu-$a.e.

The proof is complete.
\end{proof}


\section{Passage to the limit as $n\to\infty$}

This section encloses the main results regarding the asymptotic behavior of problem \eqref{ec.gn} as $n\to\infty$. First of all, we find a priori estimates ensuring the existence of such a limit.

We begin with an estimate of the Orlicz-Sobolev norm for solutions to \eqref{ec.g}.
\begin{prop} \label{prop.cota}
Let $G$ be an Orlicz function satisfying \eqref{cond}. Let $r\in [(p^-)',\infty]$ and $f\in L^r(\Omega)$.

Let $u\in W^{s,G}_0(\Omega)$ be the weak solution of  \eqref{ec.g}. Then, there exists a positive constant $C$ depending on $N, s, |\Omega|, \d(\Omega), \|f\|_r, r, p^-$ and $p^+$ such that
$$
[u]_{s,G} \leq C.
$$
\end{prop}

\begin{proof}
We can assume that $[u]_{s,G}\geq 1$.  $u$ is a weak solution of \eqref{ec.g}, thus
$$
\langle (-\Delta_g)^s u, u\rangle = \int_\Omega fu  \,dx.
$$

Condition \eqref{cond} and the integration by parts formula \eqref{partes} yield
\begin{align}\label{izquierdo}
\begin{split}
\langle (-\Delta_g)^s u,u \rangle &=  \iint_{\R^N\times\R^N} g(D^s u)  D^s u \, d\mu\\
&\geq p^- \iint_{\R^N\times\R^N}  G(D^s u) \, d\mu\\
&= p^- \Phi_{s,G}(u) \ge p^- [u]_{s,G}^{p^-}.
\end{split}
\end{align}

On the  other hand,  Lemma \ref{lema.inclusion}, H\"older's inequality for Orlicz functions \eqref{HolderG} and Corollary \ref{poincare.norma} imply that
\begin{align*}
 \int_\Omega f u \, dx &\leq 2 \|f\|_{G^*} \|u\|_{G}\leq 2C(|\Omega|, p^-,p^+,r) \|f\|_r C\d^s   [u]_{s,G},
\end{align*}
where $C(|\Omega|, p^-,p^+,r)$ is the constant specified in Lemma \ref{lema.inclusion}, $C$ is the constant appearing in Corollary \ref{poincare.norma} and $\d=\d(\Omega)$ stands for the diameter of $\Omega$.

This last estimate finishes the proof of the proposition.
\end{proof}

Some remarks are in order:

\begin{rem}
Assume that  $\{G_n\}_{n\in\N}$ is  a sequence of Orlicz functions for which \eqref{cond} holds with $p_n^-\to\infty$ as $n\to\infty$. Furthermore, suppose that there exists $\beta>1$ such that $p^+_n\le \beta p^-_n$.

Then, the complementary sequence $\{G_n^*\}_{n\in\N}$ verifies \eqref{cond*}. Notice that $(p^-_n)'\to 1$ as $n\to\infty$. As a result if $f\in L^r(\Omega)$ for some $r>1$,  there exists $n_0\in\N$ such that $f\in L^{G^*_n}(\Omega)$ for every $n\ge n_0$.
\end{rem}

\begin{rem}\label{Cn1}
Fix $r>1$ and let $u_n\in W^{s,G_n}_0(\Omega)$ be the weak solution to \eqref{ec.gn}. From the explicit computation of the constants related to Lemma \ref{lema.inclusion} and Corollary \ref{poincare.norma} it is straightforward to check that the estimate in Proposition \ref{prop.cota}
$$
[u_n]_{s,G_n}\le C_n,
$$
verifies that $C_n\to 1$ as $n\to\infty$. In particular, $[u_n]_{s,G_n}$ is bounded independently of $n\in\N$.
\end{rem}

As a consequence of the bound in Proposition \ref{prop.cota} combined  with the Sobolev immersion, we can deduce uniform H\"older estimates for the weak solutions of \eqref{ec.g}. We exploit this fact in the next result.

\begin{prop}\label{cota}
Let $f\in L^r(\Omega)$ with $r>1$ and let $u_n$ be the corresponding weak solution to \eqref{ec.gn}. Then, there exists a subsequence $\{u_{n_k}\}_{k\in\N}\subset \{u_n\}_{n\in\N}$ and a function $u_\infty\in C^{0,s}(\R^N)$ such that $u_{n_k} \to u_\infty$ uniformly in $\R^N$ as $k\to\infty$. Moreover, $u_\infty=0$ in $\R^N\setminus \Omega$ and
$$
[u_\infty]_{C^{0,s}}:= \sup_{\substack{x,y\in\R^N \\ x\neq y}} \frac{|u_\infty(x)-u_\infty(y)|}{|x-y|^s} = \|D^s u_\infty\|_\infty \leq 1.
$$
\end{prop}

\begin{proof}
Let $0<t<s<1$ and $q>1$ such that $tq>N$. Hence, there exists $n_0\in\N$ such that $p_n^->q$ for every $n\ge n_0$.  Corollary \ref{cor.inmersion} guarantees  that the sequence $\{u_n\}_{n\ge n_0}$ is bounded in $W^{t,q}_0(\Omega)$.

The usual embedding for  fractional Sobolev spaces ensures that $W^{t,q}_0(\Omega)\subset C^{0,\alpha}(\R^N)$ with $\alpha = t - \frac{N}{q}$ continuously. Hence, by Arzela-Ascoli's theorem, there exists a subsequence (still denoted by $\{u_n\}_{n\ge n_0}$) and a function $u_\infty\in C_b(\R^N)$ such that $u_n\to u_\infty$ uniformly on compact sets.

The fact that $u_n = 0$ in $\R^N\setminus \Omega$ for every $n\in \N$, implies  that indeed $u_\infty = 0$ in $\R^N\setminus \Omega$ and hence the convergence $u_n\to u_\infty$ is uniformly on $\R^N$. Moreover, $u_\infty \in C^{0,\alpha}(\R^N)$.

Let us now show that $u_\infty$ actually belongs to $C^{0,t}(\R^N)$. To this end, we need to estimate $\|D^t u_\infty\|_{L^\infty(\Omega\times\Omega)}$.

Take some $M<\|D^t u_\infty\|_{L^\infty(\Omega\times\Omega)}$ and consider the set
$$A=\{(x,y)\in \Omega\times\Omega\colon |D^t u_\infty(x,y)|>M\}.$$
Since $D^t u_\infty \in L^q(\Omega\times\Omega, d\mu)$, it follows that $0<\mu(A) <\infty$
(observe that a set has $\mu-$measure zero if and only if it has zero Lebesgue measure).
Then,
$$
[u_\infty]_{t,q;\Omega}^q \ge \iint_A |D^t u_\infty|^q\, d\mu\ge M^q \mu(A).
$$
This last inequality holds for every $q>\frac{N}{t}$, hence, we can pass to the limit $q\to\infty$ and obtain
$$
M\le \liminf_{q\to\infty} [u_\infty]_{t,q;\Omega}.
$$
Since $M<\|D^t u_\infty\|_\infty$ is arbitrary, we get
\begin{equation}\label{infti-t}
\|D^t u_\infty\|_\infty\le \liminf_{q\to\infty} [u_\infty]_{t,q;\Omega}.
\end{equation}
On the other hand, recall that $u_n\to u_\infty$, thus by Fatou's lemma we have
\begin{equation}\label{untq}
[u_\infty]_{t,q;\Omega}\le \liminf_{n\to\infty}  [u_n]_{t,q;\Omega}.
\end{equation}
Proposition \ref{prop.reg} leads now to
$$
[u_n]_{t,q;\Omega}\le C(\Omega, N, q, s-t) [u_n]_{s, G_n}.
$$
We pass to the limit using Remark \ref{Cn1}, to obtain
\begin{equation}\label{untq2}
\liminf_{n\to\infty}  [u_n]_{t,q;\Omega} \le C(\Omega, N, q, s-t).
\end{equation}

Combining \eqref{infti-t}, \eqref{untq} and \eqref{untq2} and recalling Remark \ref{remark.clave}, it gives
\begin{align*}
\|D^t u_\infty\|_\infty &\le \liminf_{q\to\infty}\liminf_{n\to\infty} [u_n]_{t,q;\Omega}\\
&\le \limsup_{q\to\infty} C(\Omega, N, q, s-t)\\
&\le \d^{s-t},
\end{align*}
where $\d=\d(\Omega)$ denotes the diameter of $\Omega$.

From this last inequality, the result  holds by taking the limit $t\to s$.
\end{proof}
Consider the Space
\begin{equation}\label{Y}
Y=\left\{\phi\in C^{0,s}(\R^N)\colon \phi=0 \text{ in } \mathbb R^N\setminus\Omega,  \text{ and } \sup_{\substack{x,y\in \R^N \\ x\neq y}}\frac{| \phi(y)-\phi(x)|}{|x-y|^s} \leq 1\right\}
\end{equation}
and observe that Proposition \ref{cota} implies that $u_\infty\in Y$.  Indeed, we have the following general estimates for $u_\infty$.
\begin{prop}
Let $u_\infty$ be a function given by Proposition \ref{cota}. Then, $u_\infty$ verifies the following estimates in viscosity sense
\begin{equation}\label{genest}
\LL_s^+ u_\infty\leq 1\quad \text{and}\quad \LL_s^-u_\infty\geq -1.
\end{equation}
\end{prop}

\begin{proof}
This proposition is precisely the content of \cite[Lemma 6.4]{FPL}. Nevertheless, we include the proof under the notation adopted in this paper for convenience of the reader.

Take first some $\varphi\in  C^1_c(\R^N)$  such that $u_\infty(x_0)=\varphi(x_0)$  and  $\varphi(y)\geq u_\infty(y)$ for all $y\in \R^N$. Taking into account that $u_\infty\in Y$, it is easily deduced that
$$
\frac{\varphi(x_0)-\varphi(y)}{ |x_0-y|^s}\le \frac{u_\infty(x_0)-u_\infty(y)}{ |x_0-y|^s}\leq 1,\quad \text{for any } y\in\R^N.
$$
In particular,
$$
\sup_{y\in\R^N}\frac{\varphi(x_0)-\varphi(y)}{ |x_0-y|^s}=\LL_s^+ \varphi(x_0)\leq 1,
$$
which proves that $\LL_s^+ u_\infty\leq 1$ in viscosity sense.

On the other hand,  we choose  $\phi\in  C^1_c(\R^N)$ verifying $u_\infty(x_0)=\phi(x_0)$  and $\phi(y)\leq u_\infty(y)$ for every $y\in \R^N$, then
$$
\phi(x_0)-\phi(y)\ge u_\infty(x_0)-u_\infty(y)\ge -|x_0-y|^s,\quad \text{for any } y\in\R^N.
$$
Therefore,
$$
\inf_{y\in\R^N}\frac{\phi(x_0)-\phi(y)}{ |x_0-y|^s}=\LL_s^- \varphi(x_0)\geq -1,
$$
which means that $u_\infty$ verifies
$$
\LL_s^- u_\infty\geq -1,
$$
in viscosity sense, and concludes the proof.
\end{proof}

In order to study the limiting equations, it will be useful to introduce the following operators: given $G$ an Orlicz function satisfying \eqref{cond} such that $p^->1/(1-s)$, $\phi\in C^1(\R^N)\cap L^\infty(\R^N)$ and $x\in\R^N$, we define the sets
\begin{equation}
\label{S+-}
S_\phi^+(x) := \{y\in\R^N\colon  D^s\phi(x,y)\geq 0\}
\end{equation}
and the operators
\begin{equation}
\label{lambda+-}
\begin{split}
&\lambda_g^+(\phi)(x)=\inf\left\{\lambda>0\colon \int _{S^+\phi(x)} g\left(\frac{D^s \phi(x,y)}{\lambda}\right)\frac1{|x-y|^{N+s}}\,dy\leq 1\right\},\\
&\lambda_g^-(\phi)(x) = -\lambda_g^+(-\phi)(x).
\end{split}
\end{equation}
Observe that Lemma \ref{pointwise} guarantees that   the operators $\lambda_g^{\pm}(\phi)$ are properly defined.

We focus now on the identification of the limit problem. With this aim in mind, we first show the following convergence result.

Let us introduce the space
$$
Z=\left\{\phi\in C^1(\R^N)\colon \lim_{|y|\to\infty} \phi(y)=0 \text{ and } \phi(\tilde y)=0 \text{ for some } \tilde y\in \R^N\right\}.
$$
Notice that $C^1_c(\R^N)\subset Z\subset C^1(\R^N)\cap L^\infty(\R^N)$.

First a technical lemma.
\begin{lema}\label{technical}
Let $f\in C(\R^N)$ be such that there exists $x_0\in \R^N$ with $f(x_0) = \lim_{|x|\to\infty} f(x)$. Then, there exists $x_1\in \R^N$ such that
$$
f(x_1) = \sup_{x\in\R^N} f(x).
$$
\end{lema}

\begin{proof}
Let $\{x_n\}_{n\in\N}\subset \R^N$ be such that $f(x_n)\to \sup_{x\in\R^N} f(x)$. In case that $\{x_n\}_{n\in\N}$ has an accumulation point $\bar x$, it follows that $f(\bar x) = \sup_{x\in\R^N} f(x)$. Else, it holds that $|x_n|\to\infty$. But then
$\sup_{x\in\R^N} f(x) = \lim_{|x|\to\infty}f(x) = f(x_0)$.
\end{proof}

\begin{lema}\label{chambolle}
Let $\{G_n\}_{n\in\N}$ be a sequence of Orlicz functions satisfying \eqref{cond} with $p^+_n\le \beta p^-_n\to\infty$ as $n\to\infty$. Let $\{x_n\}_{n\in\N}\subset \R^N$ be such that $x_n\to x_0$ as $n\to\infty$. Then, for every $\phi\in Z$ we have that
$$
\lambda_{g_n}^\pm(\phi)(x_n)\to \LL_s^\pm \phi(x_0)\qquad \text{as } n\to\infty,
$$
where the operators $\LL_s^\pm$ are given in \eqref{Linfty}.
\end{lema}

\begin{proof} We just analyze the convergence $\lambda_{g_n}^+(\phi)(x_n)\to \LL_s^+ \phi(x_0)$, since the proof of the second limit follows from the first one, observing that $\LL_s^-(\phi) = -\LL_s^+(-\phi)$.

 Notice that $\LL_s^+\phi(x_0)>0$. Then, for any $0<t<\LL_s^+\phi(x_0)$ there exists $y_0\in\R^N$ verifying $D^s\phi(x_0,y_0)>t$. The fact that $\phi\in C^1(\R^N)$ guarantees the continuous extension of  $D^s\phi$  to $\R^N\times\R^N$, since $D^s\phi(x,x)=0$. Hence there exists $\delta>0$ such that $B_\delta(y_0)\subset \{D^s\phi(x_0,y)>t\}$. Furthermore, this in particular means that $x_0\not\in B_\delta(y_0)$. In addition, by continuity
$D^s\phi(x_n,y)\to D^s\phi(x_0,y)$ as $n\to\infty$, there exists $n_0\in\N$ such that if $n\ge n_0$, then $B_\delta(y_0)\subset \{D^s\phi(x_n,y)>t\}$.

Consequently, by the definition of $\lambda_{g_n}^+(\phi)(x_n)$ and the above considerations
\begin{align*}
1 &= \int_{S_\phi^+(x_n)}g_n\left(\frac{D^s \phi(x_n,y)}{\lambda_{g_n}^+(\phi)(x_n)}\right)\frac1{|x_n-y|^{N+s}}\, dy \\
& \geq \int_{B_\delta(y_0)}g_n\left(\frac{D^s \phi(x_n,y)}{\lambda_{g_n}^+(\phi)(x_n)}\right)\frac1{|x_n-y|^{N+s}}\, dy\\
& \geq g_n\left(\frac{t}{\lambda_{g_n}^+(\phi)(x_n)}\right)\int_{B_\delta(y_0)}\frac1{|x_n-y|^{N+s}}\, dy\\
& \geq Cg_n\left(\frac{t}{\lambda_{g_n}^+(\phi)(x_n)}\right),
\end{align*}
where $C>0$ depends on $N, s, \delta$ and $|x_0-y_0|$, tough it is independent on $n\in\N$.

Assume now that $t<\lambda_{g_n}^+(\phi)(x_n)$. By \eqref{cotag2},
$$
g_n\left(\frac{t}{\lambda_{g_n}^+(\phi)(x_n)}\right)\geq p_n^- \left(\frac t{\lambda_{g_n}^+(\phi)(x_n)}\right)^{p_n^+-1},
$$
and as a result,
\begin{equation}\label{t<}
\lambda_{g_n}^+(\phi)(x_n)\geq t(C p^-_n)^\frac{1}{p_n^+ - 1}.
\end{equation}
In case $t\ge \lambda_{g_n}^+(\phi)(x_n)$, we can infer analogously that
\begin{equation}\label{t>}
\lambda_{g_n}^+(\phi)(x_n)\geq t(C p^-_n)^\frac{1}{p_n^- - 1}.
\end{equation}
  Inequalities  \eqref{t<} and \eqref{t>}, combined with condition \eqref{cond.beta} lead to
$$
\liminf_{n\to\infty}\lambda_{g_n}^+(\phi)(x_n)\geq t.
$$
This proves that
\begin{equation}\label{lowerL}
\liminf_{n\to\infty}\lambda_{g_n}^+(\phi)(x_n)\geq\LL_s^+\phi(x_0).
\end{equation}

Next we show the upper estimate.

Since  $\lim_{|y|\to\infty}\phi(y)=0=\phi(\tilde y)$, Lemma \ref{technical} ensures the existence of  $y_1\in\mathbb R^N$ such that
$$
\sup_{y\in \mathbb R^N} (\phi(y)-\phi(x_n))_+=(\phi(y_1)-\phi(x_n))_+.
$$

We first write
\begin{align*}
1=& \int _{S_\phi^+(x_n)\cap B_1(x_n)} g_n\left(\frac{D^s\phi(x_n,y)}{\lambda_n^+(\phi)(x_n)}\right) \frac{1}{|y-x_n|^{N+s}}\,dy\\
&+ \int _{S_\phi^+(x_n)\setminus B_1(x_n)} g_n\left(\frac{D^s\phi(x_n,y)}{\lambda_n^+(\phi)(x_n)}\right) \frac{1}{|y-x_n|^{N+s}}\,dy
= I_n+II_n.
\end{align*}
This in particular means that either $I_n\geq 1/2$ or $II_n\geq 1/2$.

Let us suppose first that $I_n\geq 1/2$ and decompose the integral according to the following sets
\begin{align*}
&A=\left\{y\in B_1(x_n)\cap S_\phi^+(x_n)\colon D^s\phi(x_n,y)\le \lambda_n^+(\phi)(x_n)\right \}\quad \text{and}\\
&B=\Big(B_1(x_n)\cap S_\phi^+(x_n)\Big)\setminus A.
\end{align*}

Denote by $L$ the Lipschitz constant for $\phi$ in $B_2(x_0)$ and choose $\{\gamma_n\}_{n\in\N}\subset \R$  such that $0<\gamma_n\to 0$ as $n\to\infty$ to be determined.

Observe that the next pointwise estimate holds
$$
|D^s\phi(x_n, y)|\le (L|x_n-y|^{1-s})^{\gamma_n} (\LL_s^+ \phi(x_n))^{1-\gamma_n}.
$$
Therefore, applying \eqref{cotag1} or \eqref{cotag2} depending whether  $y\in A$ or $y\in B$, we arrive at
\begin{equation}\label{cotaAB}
g_n\left(\frac{D^s\phi(x_n,y)}{\lambda_{g_n}^+(\phi)(x_n)}\right) \le p_n^+ (L|x_n-y|^{1-s})^{\gamma_n(p_n^\pm-1)} \left(\frac{(\LL_s^+ \phi(x_n))^{1-\gamma_n}}{\lambda_{g_n}^+(\phi)(x_n)}\right)^{p_n^\pm-1}.
\end{equation}
Without loss of generality, it can be assumed that $L>1$. Furthermore, suppose that
$$
\frac{(\LL_s^+ \phi(x_n))^{1-\gamma_n}}{\lambda_{g_n}^+(\phi)(x_n)}\ge 1,
$$
the other case being analogous. Thus, by \eqref{cotaAB},
\begin{align*}
\frac12 \leq I_n \le& p_n^+ L^{\gamma_n (p_n^+-1)} \left(\frac{(\LL_s^+ \phi(x_n))^{1-\gamma_n}}{\lambda_{g_n}^+(\phi)(x_n)}\right)^{p_n^+-1}\\
&\times \int_{B_1(x_n)} \frac{|x_n-y|^{(1-s)\gamma_n(p_n^+-1)} + |x_n-y|^{(1-s)\gamma_n(p_n^- -1)}}{|x_n-y|^{N+s}}\, dy.
\end{align*}
An easy computation shows that if we take for instance $\gamma_n :=\frac{1+s}{(1-s)(p_n^+-1)}$,  the integral above is bounded by a constant, which only depends  on the dimension $N$.

Consequently,
$$
\lambda_{g_n}^+(\phi)(x_n)\le C(N)^{\frac{1}{p_n^+-1}} (p_n^+)^{\frac{1}{p_n^+-1}} L^{\gamma_n} (\LL_s^+\phi(x_n))^{1-\gamma_n}.
$$

This shows that
\begin{equation}\label{cotasup}
\limsup_{n\to\infty} \lambda_{g_n}^+(\phi)(x_n)\leq \LL_s^+\phi(x_0).
\end{equation}

Then, \eqref{lowerL} together with \eqref{cotasup} imply the result.

In case $I_n<\frac12$ and thus $II_n\ge \frac12$, the argument is completely analogous and is left to the reader.
\end{proof}

We have all of the ingredients to identify the limit problem, for which the influence of the function $f$   is only through its support and sign. As advanced, we get the same limit than in the Fractional $p$-Laplacian case, see \cite{FPL}.

\begin{thm}\label{teo.pblim}
Let $f =f(x)\in C(\overline{\Omega})$. A function $u_\infty\in Y$ obtained as a uniform limit of a subsequence of $\{u_n\}_{n\in\N}$,  is a viscosity solution of the problem
$$
\begin{cases}
\LL_s^+ u_\infty = 1 & \text{in }\{ f>0\},\\
\LL_s^- u_\infty = -1 & \text{in }\{ f<0\},\\
\LL_s u_\infty = 0 & \text{in }\Omega\setminus \supp(f)^o,\\
\LL_s u_\infty \ge 0 & \text{in } \Omega \cap\partial\{ f>0\}\setminus \partial\{f<0\},\\
\LL_s u_\infty \le 0 & \text{in } \Omega \cap\partial\{f<0\}\setminus \partial\{f>0\}.
\end{cases}
$$
\end{thm}

\begin{proof} We determine the different equations in viscosity sense depending on the sign of the function $f$.

\medskip

\noindent {\bf 1.} $\{f>0\}$. The general estimates  \eqref{genest} imply in particular that  $u_\infty$ is a viscosity subsolution to $\mathcal L^+_s u_\infty= 1$.

 Let us see that in fact $u_\infty$ is a viscosity supersolution.  Take a test function, $\phi$ touching $u_\infty$ from below at a point $x_0\in \{f>0\}$. By Remark \ref{rem.cte} there is no loss of generality to admit   that the function $u_\infty-\phi$ is strictly positive (adding some constant to the test function) and it  attains  a strict minimum at $x_0$. The uniform convergence shown in Proposition \ref{cota} and the boundedness of $\Omega$ allow us to extract, up to subsequences,  $x_n\to x_0$, where $x_n$ are  points at which the positive function $u_{n}-\phi$ reaches a minimum. Moreover, for $n$ large  we know that indeed  $u_{n}$ are viscosity supersolutions by Lemma \ref{debilesvisco}. Consequently, by the oddity of $g_n$ we can write
 \begin{align*}
 (-\Delta_{g_n})^s \phi_n(x_n) = &\int _{S_\phi^+(x_n)}g_n(D^s\phi(x_n,y))\frac{dy}{|x_n-y|^{N+s}}\\
 & + \int _{S_{-\phi}^+(x_n)}g_n(D^s\phi(x_n,y))\frac{dy}{|x_n-y|^{N+s}}\\
 \geq& f(x_n)>0,
 \end{align*}
 where the sets $S_\phi^+(x)$ are specified in \eqref{S+-}.
\begin{align*}
(-\Delta_{g_n})^s \phi_n(x_n) = &\int _{S_\phi^+(x_n)}g_n(D^s\phi(x_n,y))\frac{dy}{|x_n-y|^{N+s}}\\
& + \int _{S_{-\phi}^+(x_n)}g_n(D^s\phi(x_n,y))\frac{dy}{|x_n-y|^{N+s}}\\
\geq& f(x_n)>0,
\end{align*}
 Since the second integral is nonpositive, this means that in fact
$$
\int _{S_\phi^+(x_n)}g_n(D^s\phi(x_n,y))\frac{dy}{|x_n-y|^{N+s}} \geq f(x_n).
$$
But then, from the definition of the operator $\lambda_g^+(\phi)$ given in \eqref{lambda+-},
$$
\lambda_{g_n}^+(\phi)(x_n)\geq\min\{f(x_n)^\frac{1}{p_n^+},f(x_n)^\frac{1}{p_n^-}\},
$$
and passing to the limit using the convergence result in Lemma~\ref{chambolle} and the continuity of $f$, it yields $\LL_s^+ \phi(x_0)\geq 1$, which implies that $u_\infty$ is a viscosity supersolution to $\mathcal L^+_s u_\infty= 1$.

\medskip

\noindent {\bf 2.} $\{f<0\}$.  This case is completely analogous to the previous one and is left to the reader.

\medskip

\noindent{\bf 3.} $\{\Omega\setminus \supp f\}^\circ$. In this case we just show that $u_\infty$ is a viscosity supersolution to $\LL_{s}u_\infty=0$, since  the proof  that $u_\infty$ is a viscosity subsolution  runs similarly.

 Choose again $\{x_n\}_{n\in\N}\subset\Omega$ such that $x_n\to x_0$, being minima for  $u_n-\phi>0$ and  such that $u_\infty-\phi>0$ attains a strict minimum at $x_0$. Notice that for $n$ sufficiently large $f(x_n)=0$, given that we are considering the interior set. Therefore, $(-\Delta_{g_n})^s\phi(x_n)\geq 0$. Moreover, using \eqref{estimareg}, one can easily conclude that, denoting $\phi_\lambda(x) = \lambda^{-1}\phi(x)$,
$$
 (-\Delta_{g_n})^s\phi_\lambda(x_n)\geq 0, \quad \text{for all } \lambda>0.
$$
Equivalently
 $$
\int _{S_\phi^+(x_n)}\! g_n\left(\frac{D^s\phi(x_n,y)}{\lambda}\right)\frac{dy}{|x_n-y|^{N+s}} + \int _{S_{-\phi}^+(x_n)} \! g_n\left(\frac{D^s\phi(x_n,y)}{\lambda}\right)\frac{dy}{|x_n-y|^{N+s}}\ge 0.
$$
Equivalently
 $$
\int _{S_\phi^+(x_n)}\! g_n\left(\frac{D^s\phi(x_n,y)}{\lambda}\right)\frac{dy}{|x_n-y|^{N+s}} + \int _{S_{-\phi}^+(x_n)} \! g_n\left(\frac{D^s\phi(x_n,y)}{\lambda}\right)\frac{dy}{|x_n-y|^{N+s}}\ge 0.
$$
But this implies that
$$
\int _{S_\phi^+(x_n)}\!\! g_n\left(\frac{D^s\phi(x_n,y)}{\lambda}\right)\frac{dy}{|x_n-y|^{N+s}}\ge \int _{S_{-\phi}^+(x_n)}\!\! g_n\left(\frac{D^s(-\phi)(x_n,y)}{\lambda}\right)\frac{dy}{|x_n-y|^{N+s}}
$$
from where it follows that
$$
\lambda_{g_n}^+(\phi)(x_n)\ge \lambda_{g_n}^+(-\phi)(x_n) = -\lambda_{g_n}^-(\phi)(x_n).
$$

Thanks to Lemma \ref{chambolle}, in the limit it gives
$$
\LL_s \phi(x_0) = \LL_s^+ \phi(x_0) + \LL_s^-\phi(x_0) \ge 0.
$$

\medskip

\noindent{\bf 4.} $\Omega \cap\partial\{ f>0\}\setminus \partial\{f<0\}$. Recall that $x_0\in \Omega \cap\partial\{ f>0\}\setminus \partial\{f<0\}$ implies that $f(x_0)=0$ and $x_0$ is approached by points $x_n$ such that $f(x_n)\ge 0$. Similar arguments as in the previous step infer that $\LL_s u_\infty\ge 0$ in viscosity sense.

In this case, we just know for subsolutions that $\LL_s^+ \varphi(x_0)\leq 1$, invoking the general estimates \eqref{genest}.

\medskip

 \noindent{\bf 5.} $\Omega \cap\partial\{
f<0\}\setminus \partial\{f>0\}$. This case is identical to the previous one.

\medskip

The proof is complete.
\end{proof}

\section{Identification of the limit. A $\Gamma-$convergence result.}

This work concludes with a result proving a property verified by the limit $u_\infty$, which will  be useful to identify it in certain cases.

Let us first recall the definition of the concept of $\Gamma$-convergence (introduced in \cite{DG, DGF}) in metric spaces. The reader
is referred to \cite{DalM} and \cite{braides} for a comprehensive introduction to the topic.
\begin{defn}\label{def}
Let $X$ be a metric space. A sequence $\{F_n\}_{n\in\N}$ of functionals $F_n\colon X\to \bar\R:=\mathbb{R}\cup \{+\infty\}$ is said to $\Gamma(X)$-converge to $F\colon X\to \bar\R$, and we write $\Gamma(X)-\lim\limits_{n\to \infty}F_n=F$, if the following hold:
\begin{description}
  \item[(i)] for every $u\in X$ and $\{u_n\}_{n\in\N}\subset X$ such that $u_n\to u$ in $X$, we have
  $$F(u)\leq \liminf_{n\to \infty}F_n(u_n)\,;$$
  \item[(ii)] for every $u\in X$ there exists a recovery sequence $\{u_n\}_{n\in\N}\subset X$ such that $u_n\to u$ in $X$ and
  $$F(u)\geq \limsup_{n\to \infty}F_n(u_n)\,.$$
\end{description}
\end{defn}
For each integer $n\geq1$ consider the functionals $I_n\colon L^1(\Omega)\to [0,\infty]$ defined by
$$
I_n(u) = \begin{cases}
\D\iint_{\R^N\times \R^N} G_n(D^s u)\, d\mu & \text{if } u \in W_0^{s, G_n}(\Omega) ,\\
+\infty & \text{otherwise}.
\end{cases}
$$

The next theorem reveals which is the $\Gamma$-limit for the sequence $\{I_n\}_{n\in\N}$.
\begin{thm}\label{t2}
Define $I_\infty\colon L^1(\Omega)\to [0, \infty]$ by
$$
I_\infty(u) = \begin{cases}
0  & \text{if } u\in Y,\\
+\infty & \text{otherwise},
\end{cases}
$$
where $Y$ is specified in \eqref{Y}. Then $\Gamma(L^1(\Omega))-\lim\limits_{n\to \infty}I_n=I_\infty$.
\end{thm}

\begin{proof}
We start by verifying the existence of a recovery sequence. If
$I_\infty(u)=\infty,$ the inequality clearly holds for any sequence $u_n
\to u$ strongly in $L^1(\Omega)$. On the other hand, if
$I_\infty(u)< +\infty$ we must have $I_\infty(u)=0$ and,
as a result, $u\in Y$. For each $n\in \mathbb{N},$ let $u_n := (1-\ve_n)u$, where $\ve_n\to 0$ is to be determined. The fact that $u\in Y$ entails that
$$
\limsup_{n\to\infty} \iint_{\R^N\times\R^N} G_n(D^s u_n)\, d\mu = 0.
$$

Indeed, since $u\in Y$ we have that $|D^s u(x,y)|\le 1$ and hence we estimate the integrand as follows:
$$
G_n(D^su_n(x,y))\le |D^s u_n(x,y)|^{p_n^-} \le (1-\ve_n)^{p_n^-} \frac{|u(x)-u(y)|}{|x-y|^s}.
$$
Thus, as long as $x,y\in\Omega$, using that $u$ is Lipschitz, we get
\begin{equation}\label{cotaGn1}
G_n(D^s u_n(x,y))\le(1-\ve_n)^{p_n^-} L |x-y|^{1-s},
\end{equation}
where $L$ is the Lipschitz constant of $u$, while if $x\in\Omega$, $y\not\in \Omega$, we find
\begin{equation}\label{cotaGn2}
G_n(D^su_n(x,y))\le(1-\ve_n)^{p_n^-} 2\|u\|_{\infty} |x-y|^{-s}
\end{equation}

Then, we decompose the integral as follows
$$
\iint_{\R^N\times\R^N} G_n(D^s u_n) \, d\mu = \iint_{\Omega\times\Omega} G_n(D^s u_n) \, d\mu + 2\iint_{\Omega\times\Omega^c} G_n(D^su_n) \, d\mu,
$$
where $\Omega^c = \R^N\setminus \Omega$.

For the first term, we invoke \eqref{cotaGn1} and for the second term \eqref{cotaGn2}. Hence,
$$
\iint_{\R^N\times\R^N} G_n(D^su_n) \, d\mu \le C(1-\ve_n)^{p_n^-}.
$$
Therefore, choosing $\ve_n>0$ such that $(1-\ve_n)^{p_n^-}\to 0$ as $n\to\infty$, the result follows.
\

To prove liminf-inequality from the definition of the $\Gamma$-convergence, there is not  loss
of generality in assuming that  $\{u_n\}_{n\in\N} \subset  W_0^{s,G_n}(\Omega)$ and
\begin{equation}\label{bound}
\liminf_{n\to \infty} I_n(u_n)=\lim_{n\to\infty} I_n(u_n)<\infty.
\end{equation}

We note that if $x_0\neq y_0$ and $0<R<|x_0-y_0|$, then $\mu(B_R(x_0,y_0))<\infty$. Moreover, if $u\in L^1(\R^N)$ then $D^s u\in L^1(B_R(x_0,y_0), d\mu)$.

Let $(x_0,y_0)\in\R^N\times \R^N$ with $x_0\neq y_0$ be a Lebesgue point for $D^su$ according to the measure $\mu$, namely
$$\lim_{r\to 0^+}\frac{1}{|B_r(x_0,y_0)|}\iint_{B_r(x_0,y_0)}|D^s u(x,y)-D^s u(x_0,y_0)|\;d\mu=0\,.$$

We fix this point $(x_0,y_0)$ and denote $B_r = B_r(x_0,y_0)\subset \R^N\times \R^N$.

On the other hand,
\begin{equation}\label{cota-D}
\iint_{B_r} |D^s u_n|\, d\mu\leq \|D^su_n\|_{L^{p_n^-}(B_r)}\mu(B_r)^\frac{p_n^--1}{p_n^-}.
\end{equation}
Furthermore,
\begin{align*}
\iint_{B_r}|D^s u_n|^{p_n^-}\,d\mu&=\iint_{B_{r,n}^-}|D^s u_n|^{p_n^-}\,d\mu+\iint_{B_{r,n}^+}|D^s u_n|^{p_n^-}\,d\mu\\
&\leq \mu(B_r)+\iint_{B_{r,n}^+} G_n(D^su_n)\,d\mu,
\end{align*}
where the sets are defined as
\begin{align*}
&B_{r,n}^-:=\{(x,y)\in B_r\colon |D^s u_n(x,y)|<1\}\text{ and }\\& B_{r,n}^+:=\{(x,y)\in B_r\colon |D^s u_n(x,y)|\geq1\}.
\end{align*}
Plugging this estimate on \eqref{cota-D} leads to
$$
\iint_{B_r}|D^s u_n|\,d\mu\leq\left(\mu(B_r)+\iint_{B_{r,n}^+} G_n(D^su_n)\,d\mu\right)^\frac1{p_n^-}
\mu(B_r)^\frac{p_n^--1}{p_n^-}
$$
The nonrestrictive assumption $\sup_{n\in\N}I_n(u_n)<\infty$ implies that
$$
\sup_{n\in\N} \iint_{B_{r,n}^+}G_n(D^su_n)\,d\mu<\infty,
$$
hence we can pass to the limit in the above estimate and deduce
\begin{equation}\label{cota-sup}
\limsup_{n\to\infty}\iint_{B_r}|D^s u_n|\,d\mu\leq \mu(B_r).
\end{equation}

Now we take some $q<p_n^-$, $0<t<s$ and observe that, from Corollary \ref{cor.inmersion}, the sequence $\{u_n\}_{n\in\N}$ is bounded in $W^{t,q}_0(\Omega)$. Therefore, from Rellich-Kondrachov Theorem, we can assume that $u_n\to u$ strongly in $L^q(\Omega)$ and a.e. As a result, $D^s u_n\to D^s u$ a.e. in $\R^N\times \R^N$.

Thus, by Fatou's Lemma,
$$
\iint_{B_r}|D^s u|\,d\mu\leq\liminf_{n\to\infty}\iint_{B_r}|D^s u_n|\,d\mu.
$$
Taking into account \eqref{cota-sup} we get
$$
\iint_{B_r}|D^s u|\,d\mu\leq \mu(B_r),
$$
or equivalently,
$$
\frac1{\mu(B_r)}\iint_{B_r}|D^s u|\,d\mu\leq 1.
$$
Recalling that $B_r = B_r(x_0,y_0)$ where $(x_0,y_0)$ with $x_0\neq y_0$ was a Lebesgue point for $D^s u$, this shows that indeed $|D^su(x_0,y_0)|\leq 1$. Then $u\in Y$ and hence
$$
I_\infty (u)=0\leq \liminf_{n\to\infty}I_n(u_n),
$$
and the proof is complete.
\end{proof}

A simple consequence of Theorem \ref{t2} is the following.
\begin{cor}\label{c1}
Let $f\in L^r(\Omega)$ for some $r>1$. For each integer $n\geq1$ consider the functionals $F_n\colon L^1(\Omega)\to \mathbb{\overline{R}}$ defined by
\begin{equation}\nonumber
F_n(u):= \begin{cases}
\D\iint_{\R^N\times\R^N} G_n(D^s u)\, d\mu - \int_\Omega fu\, dx , & \text{if }\; u \in W_0^{s, G_n}(\Omega) ,\\
+\infty, & \text{otherwise}
\end{cases}
\end{equation}
Denote as $F_\infty\colon L^1(\Omega)\to \bar\R$ the functional
$$
F_\infty(u) := \begin{cases}
\D -\int_\Omega fu\, dx & \text{if } u\in Y,\\
+\infty & \text{otherwise},
\end{cases}
$$
where $Y$ is specified in \eqref{Y}. Then $\Gamma(L^1(\Omega))-\lim\limits_{n\to \infty} F_n= F_\infty$.
\end{cor}

We rewrite the previous consequence as the following maximization problem.

\begin{cor}\label{lem.pbmax} Let $f\in L^r(\Omega)$ with $r>1$,  and let $u_\infty$ be a function given by Proposition \ref{cota}. Then, $u_\infty$ satisfies
\begin{equation}\label{max}
\max_{\psi\in Y}\int_{\R^N}f\psi dx=\int_{\R^N}fu_\infty dx.
\end{equation}

\end{cor}

When $f(x)>0$ (resp. $f(x)<0$) for all $x\in\Omega$, we can identify the limit function $u_\infty$ in terms of the distance to the boundary $\partial\Omega$, thanks to the fact that $u_\infty$ solves the maximization problem \eqref{max}.  This also occurs for the local $p$-Laplacian and some related operators of variable exponent type, see \cite{BDM,PLR, PL} and references therein. Namely,
\begin{cor}
If $f(x)>0$ in $\Omega$, then there exists a unique maximizer of problem \eqref{max} given by
$$
u_\infty(x) = \begin{cases}
(d(x,\partial\Omega))^s & \text{if } x\in\Omega\\
0 & \text{if } x\in \R^N\setminus\Omega.
\end{cases}
$$
\end{cor}

\begin{proof}
The proof can be found in \cite[Lemma 6.6]{FPL}. We include the details for completeness.

First, observe that $(d(x,\partial \Omega))^s$ is an admissible function in \eqref{max} for any $s\in(0,1)$. Recall that $d(x,\partial \Omega)$ is a Lipschitz function with unit  constant, thus
$$
|(d(x,\partial \Omega))^s-(d(y,\partial \Omega))^s| \leq |d(x,\partial \Omega) - d(y,\partial \Omega) |^s \leq |x-y|^s.
$$
Hence, the corollary  follows if for any $v\in Y$ we see that $v(x)\leq (d(x,\partial \Omega))^s$. Indeed, since $v\in Y$, it holds that $|v(x)|\leq |x-y|^s$ for all $y\in\partial \Omega$, and then
$$
|v(x)|\leq \inf_{y\in \partial\Omega} |x-y|^s=(d(x,\partial \Omega))^s,
$$
which concludes the proof.
\end{proof}

\section*{Acknowledgements}

This paper is partially supported by grants UBACyT 20020130100283BA, CONICET PIP 11220150100032CO and ANPCyT PICT 2012-0153.

All of the authors are members of CONICET.


\bibliographystyle{amsplain}
\bibliography{biblio}

\end{document}